\documentclass[11pt]{amsart}
\usepackage{latexsym,amsmath,amssymb}
\usepackage[all,cmtip]{xy}

\title[Sobolev mappings, degree and homotopy classes]
{{\protect{Sobolev mappings, degree, homotopy classes and rational homology spheres}}}

\author{Pawe\l{} Goldstein}
\address{Institute of Mathematics,\newline \indent Faculty of Mathematics,
Informatics and Mechanics, \newline \indent
University of Warsaw \newline \indent Banacha 2, 02-097 Warsaw, Poland}
\email{goldie@mimuw.edu.pl}
\author{Piotr Haj\l{}asz}
\address{Department of Mathematics, University of Pittsburgh,
\newline \indent 301 Thackeray Hall, Pittsburgh,
Pennsylvania 15260}
\email{hajlasz@pitt.edu}

\thanks{P.G. was partially supported by the Polish Ministry of
Science grant no. N N201 397737 (years 2009-2012)
and P.H. was partially supported by the NSF grant DMS-0900871.}



plus 6pt minus 12pt
\DeclareMathOperator{\Hom}{Hom}
\DeclareMathOperator{\coker}{coker}

\def\eps{\varepsilon}

\newtheorem{theorem}{Theorem}
\newtheorem{lemma}[theorem]{Lemma}
\newtheorem{corollary}[theorem]{Corollary}
\newtheorem{proposition}[theorem]{Proposition}

\theoremstyle{definition}
\newtheorem{remark}[theorem]{Remark}
\newtheorem{definition}[theorem]{Definition}
\newtheorem{example}[theorem]{Example}

\newcommand{\barint}{
\rule[.036in]{.12in}{.009in}\kern-.16in \displaystyle\int }

\newcommand{\barcal}{\mbox{$ \rule[.036in]{.11in}{.007in}\kern-.128in\int $}}

\def\eqn#1$$#2$${\begin{equation}\label#1#2\end{equation}}
\def\vint#1_#2{-\kern-#1pt\int_{#2}}

\newcommand{\bbbz}{\mathbb Z}
\newcommand{\bbbr}{\mathbb R}

\newcommand{\bbbc}{\mathbb C}

\def\mvint_#1{\mathchoice
          {\mathop{\vrule width 6pt height 3 pt depth -2.5pt
                  \kern -8pt \intop}\nolimits_{\kern -3pt #1}}%
          {\mathop{\vrule width 5pt height 3 pt depth -2.6pt
                  \kern -6pt \intop}\nolimits_{#1}}%
          {\mathop{\vrule width 5pt height 3 pt depth -2.6pt
                  \kern -6pt \intop}\nolimits_{#1}}%
          {\mathop{\vrule width 5pt height 3 pt depth -2.6pt
                  \kern -6pt \intop}\nolimits_{#1}}}

\numberwithin{theorem}{section} \numberwithin{equation}{section}

\begin{document}


\sloppy

\begin{abstract}
In the paper we investigate the degree and the homotopy theory of
Orlicz-Sobolev mappings $W^{1,P}(M,N)$ between manifolds, where
the Young function $P$ satisfies a divergence condition and forms
a slightly larger space than $W^{1,n}$, $n=\dim M$. In particular,
we prove that if $M$ and $N$ are compact oriented manifolds
without boundary and $\dim M=\dim N=n$, then the degree is well
defined in $W^{1,P}(M,N)$ if and only if the universal cover of
$N$ is not a rational homology sphere, and in the case $n=4$, if
and only if $N$ is not homeomorphic to $S^4$.
\end{abstract}

\maketitle

\section{Introduction}

Let $M$ and $N$ be compact smooth Riemannian manifolds without
boundary. We consider the space of Sobolev mappings between manifolds defined
in a usual way: we assume that $N$ is isometrically embedded in
a Euclidean space $\bbbr^k$ and define $W^{1,p}(M,N)$ to be the
class of Sobolev mappings $u\in W^{1,p}(M,\bbbr^k)$ such that
$u(x)\in N$ a.e. The space $W^{1,p}(M,N)$ is a subset of a Banach
space $W^{1,p}(M,\bbbr^k)$ and it is equipped with a metric
inherited from the norm of the Sobolev space. It turns out that
smooth mappings $C^\infty(M,N)$ are not always dense in
$W^{1,p}(M,N)$ and we denote by $H^{1,p}(M,N)$ the closure of
$C^\infty(M,N)$ in the metric of $W^{1,p}(M,N)$. A complete
characterization of manifolds $M$ and $N$ for which smooth
mappings are dense in $W^{1,p}(M,N)$, i.e.
$W^{1,p}(M,N)=H^{1,p}(M,N)$, has recently been obtained by
Hang and Lin, \cite{hangl2}.

The class of Sobolev mappings between manifolds plays a central role in applications to
geometric variational problems and deep connections to algebraic topology have been
investigated recently, see e.g.
\cite{bethuel},
\cite{bethuelcdh},
\cite{bethuelz},
\cite{brezisl},
\cite{giaquintams1},
\cite{giaquintams2},
\cite{hajlasz1},
\cite{hajlasz2},
\cite{hajlasz3},
\cite{HIMO},
\cite{hang},
\cite{hang2},
\cite{hangl1},
\cite{hangl2},
\cite{heleinw},
\cite{pakzad},
\cite{pakzadr}.

In the borderline case, $p=n=\dim M$, it was proved by
Schoen and Uhlenbeck (\cite{schoenu1}, \cite{schoenu2})
that the smooth mappings
$C^{\infty}(M,N)$ form a dense subset of $W^{1,n}(M,N)$, i.e.
$W^{1,n}(M,N)=H^{1,n}(M,N)$,
and White \cite{white1}
proved that the homotopy classes are well defined in $W^{1,n}(M,N)$, see also
\cite{brezisn1}, \cite{brezisn2}.
Indeed, he proved
that for every $u\in W^{1,n}(M,N)$, if two smooth mappings $u_1, u_2\in C^{\infty}(M,N)$
are sufficiently close to $u$ in the Sobolev norm, then $u_1$ and $u_2$ are homotopic.

If $\dim M=\dim N=n$ and both manifolds are oriented,
then the degree of smooth mappings is well defined. There are several equivalent definitions
of the degree and here we consider the one that involves
integration of differential forms. If $\omega$ is a volume form on
$N$, then for $f\in C^\infty(M,N)$ we define
$$
\deg f=\left(\int_M f^*\omega\right)\bigg/\left(\int_N\omega\right)\, .
$$
Since $f^*\omega$ equals the Jacobian $J_f$ of $f$ (after identification of $n$-forms with
functions),
it easily follows from H\"older's inequality that the degree is continuous in the Sobolev
norm $W^{1,n}$. Finally, the density of smooth mappings in $W^{1,n}(M,N)$ allows us to extend
the degree continuously and uniquely to $\deg:W^{1,n}(M,N)\to\bbbz$.

Results regarding degree and homotopy classes do not, in general, extend to the case of
$W^{1,p}(M,N)$ mappings when $p<n$. This stems from the fact that the radial projection mapping
$u_0(x)=x/|x|$ belongs to $W^{1,p}(B^n,S^{n-1})$ for all $1\leq p<n$.

A particularly interesting class of Sobolev mappings to which we can extend the degree and
the homotopy theory is the Orlicz-Sobolev space $W^{1,P}(M,N)$,
where the Young function $P$ satisfies the so called divergence condition
\begin{equation}
\label{E1}
\int_1^\infty \frac{P(t)}{t^{n+1}}\, dx =\infty\, .
\end{equation}
In particular, $P(t)=t^n$ satisfies this condition, so the space $W^{1,n}(M,N)$
is an example. However, we want the space to be slightly
larger than $W^{1,n}$, and hence we also assume that
\begin{equation}
\label{E1.5}
P(t)=o(t^n)
\quad
\mbox{as $t\to\infty$.}
\end{equation}
In addition to the conditions described here
we require some technical assumptions about $P$; see
Section~\ref{Orlicz}.
Roughly speaking, $u\in W^{1,P}(M,N)$ if $\int_M P(|Du|)<\infty$.
A typical Orlicz-Sobolev space $W^{1,P}$ discussed here contains
$W^{1,n}$, and it is contained in all
spaces $W^{1,p}$ for $p<n$
$$
W^{1,n}\subset W^{1,P}\subset\bigcap_{1\leq p<n} W^{1,p}\, .
$$
A fundamental example of a Young function
that satisfies all required conditions, forms a strictly larger space than $W^{1,n}$,
and strictly smaller than the intersection of all $W^{1,p}$, $1\leq p<n$,
is
$$
P(t)=\frac{t^n}{\log(e+t)}\, .
$$
It is easy to check that $u_0\not\in W^{1,P}(B^{n}, S^{n-1})$ if and only if the
condition \eqref{E1} is satisfied, see \cite[p. 2]{HIMO},
so the presence of the divergence condition
is necessary and sufficient for the exclusion of mappings like $x/|x|$ that can
easily cause topological problems.

The class of Orlicz-Sobolev mappings under the divergence condition has been investigated
in connections to nonlinear elasticity, mappings of finite distortion \cite{IwaniecM} and
degree theory \cite{GrecoISS}, \cite{HIMO}.
Roughly speaking, many results that are true for $W^{1,n}$ mappings have counterparts
in $W^{1,P}$ as well.
In particular, it was proved in \cite{HIMO} that smooth mappings $C^{\infty}(M,N)$
are dense in $W^{1,P}(M,N)$ if $P$ satisfies \eqref{E1}, see Section~\ref{Orlicz}.

We say that a compact connected manifold $N$ without boundary, \mbox{$\dim N=n$,}
is a {\em rational homology sphere}, if it has the same deRham cohomology groups as the
sphere $S^n$.
It has been proved in \cite{HIMO} that if $M$ and $N$ are
smooth compact oriented $n$-dimensional manifolds without boundary and $N$
is {\em not} a rational homology sphere, then
the degree, originally defined on a class of smooth mappings $C^\infty(M,N)$, uniquely
extends to
a continuous function $\deg:W^{1,P}(M,N)\to\bbbz$, see also \cite{GrecoISS} for earlier results.
This is not obvious, because the Jacobian of a $W^{1,P}(M,N)$ mapping is not necessarily
integrable and we cannot easily use estimates of the Jacobian, like in the case of $W^{1,n}$
mappings, to prove continuity of the degree.
The proof given in \cite{HIMO} (cf. \cite{GrecoISS}) is not very geometric and it is based on
estimates of integrals of differential forms, Hodge decomposition, and the study of
the so called Cartan forms. Surprisingly, it has also been shown in \cite{HIMO} that
the degree is not continuous in $W^{1,P}(M,S^n)$.
Thus the results of \cite{HIMO} provide a good understanding of the situation when $N$ is
not a rational homology sphere or when $N=S^n$. There is, however, a large class
of rational homology spheres which are not homeomorphic to $S^n$ and it is natural to
ask what happens for such manifolds.

In this paper we give a complete answer to this problem.
In our proof we do not rely on methods of \cite{HIMO} and we provide a new geometric
argument that avoids most of the machinery of differential forms developed in \cite{HIMO}.
\begin{theorem}
\label{T3}
Let $M$ and $N$ be compact oriented $n$-dimensional
Riemannian manifolds, $n\geq 2$,
without boundary, and let $P$ be
a Young function satisfying conditions \eqref{E1}, \eqref{E1.5},
\eqref{doubling cond} and \eqref{growth cond 2}.
Then the degree is well defined, integer valued
and continuous in the space $W^{1,P}(M,N)$
if and only if the universal cover of $N$ is not a rational homology
sphere.
\end{theorem}
The theorem should be understood as follows: if the universal cover of $N$
is not a rational homotopy sphere, then $\deg:C^\infty(M,N)\to\bbbz$
is continuous in the norm of $W^{1,P}$ and since smooth mappings are dense in
$W^{1,P}(M,N)$ the degree uniquely and continuously extends to
$\deg:W^{1,P}(M,N)\to\bbbz$.
On the other hand, if the universal cover of $N$ is a rational homology sphere,
then there is a sequence of smooth mappings $u_k\in C^\infty(M,N)$ with
$\deg u_k=d>0$ such that $u_k$ converges to a constant mapping (and hence of degree zero)
in the norm of $W^{1,P}$.

If $n=4$, Theorem~\ref{T3} together with
Proposition~\ref{escztery} give

\begin{corollary}
\label{T3.5}
Let $M$ and $N$ be compact oriented $4$-dimensional Riemannian manifolds
without boundary and let $P$ be
a Young function satisfying conditions \eqref{E1}, \eqref{E1.5},
\eqref{doubling cond} and \eqref{growth cond 2} with $n=4$.
Then the degree is well defined, integer valued
and continuous in the space $W^{1,P}(M,N)$
if and only if $N$is not homeomorphic to $S^4$.
\end{corollary}

The next result concerns the definition of homotopy classes in $W^{1,P}(M,N)$.
\begin{theorem}
\label{T4}
Let $M$ and $N$ be two compact Riemannian manifolds
without boundary, $n=\dim M\geq 2$, and let $P$ be
a Young function satisfying the conditions \eqref{E1},
\eqref{E1.5}, \eqref{doubling cond} and \eqref{growth cond 2}.
If $\pi_n(N)=0$, then the homotopy classes are well defined in $W^{1,P}(M,N)$.
If $\pi_n(N)\neq 0$, the homotopy classes cannot be
well defined in $W^{1,P}(S^n,N)$.
\end{theorem}
Theorem~\ref{T4} has been announced in \cite[p.5]{HIMO}, but no details of
the proof have been provided.
It should be understood in a similar way as Theorem~\ref{T3}.
Let $\pi_n(N)=0$. Then for every
$u\in W^{1,P}(M,N)$ there is $\eps>0$ such that if $u_1,u_2\in C^\infty(M,N)$,
$\Vert u-u_i\Vert_{1,P}<\eps$, $i=1,2$, then $u_1$ and $u_2$ are homotopic.
Since smooth mappings are dense in $W^{1,P}(M,N)$, homotopy classes of smooth
mappings can be uniquely extended to homotopy classes of $W^{1,P}(M,N)$
mappings.
On the other hand, if $\pi_n(N)\neq 0$, there is a sequence of smooth mappings
$u_k\in C^\infty(S^n,N)$ that converges to a constant mapping in the norm of
$W^{1,P}$ and such that the mappings $u_k$ are not homotopic to a constant mapping.

The condition $\pi_n(N)=0$ is not necessary: any two continuous mappings
from $\mathbb{CP}^2$ to $\mathbb{CP}^1$ are homotopic (in particular -- homotopic to
a constant mapping), even though $\pi_4(\mathbb{CP}^1)=\pi_4(S^2)=\mathbb{Z}_2$.

Theorems~\ref{T3} and~\ref{T4} will be proved from corresponding results for
$W^{1,p}$ mappings.
\begin{theorem}
\label{T1}
Let $M$ and $N$ be two compact Riemannian manifolds without boundary,
$\dim M=n$, $n-1\leq p<n$. Then the homotopy classes are well defined in $H^{1,p}(M,N)$
if $\pi_n(N)=0$, and they cannot be well defined in $H^{1,p}(S^n,N)$,
if $\pi_n(N)\neq 0$. If $\pi_{n-1}(N)=\pi_n(N)=0$, then the homotopy classes
are well defined in $W^{1,p}(M,N)$.
\end{theorem}
The theorem should be understood in a similar way as Theorem~\ref{T4}.
The last part of the theorem follows from the results of Bethuel \cite{bethuel}
and Hang and Lin \cite{hangl2} according to which
the condition $\pi_{n-1}(N)=0$ implies density of
smooth mappings in $W^{1,p}(M,N)$, so $H^{1,p}(M,N)=W^{1,p}(M,N)$.
\begin{theorem}
\label{T2}
Let $M$ and $N$ be compact connected oriented $n$-dimensional
Riemannian manifolds without
boundary, $n-1\leq p<n$, $n\geq 2$. Then the degree is well defined,
integer valued and continuous in the space
$H^{1,p}(M,N)$ if and only if the universal cover of $N$ is not a rational homology sphere.
If the universal cover of $N$ is not a rational homology sphere and $\pi_{n-1}(N)=0$,
then degree is well defined, integer valued and continuous in $W^{1,p}(M,N)$.
\end{theorem}
Again, Theorem~\ref{T2} should be understood in a similar way as Theorem~\ref{T3} and
the last part of the theorem follows from density
of smooth mappings in $W^{1,p}$, just like in the case of Theorem~\ref{T1}.

As a direct consequence of Theorem~\ref{T2} and Proposition~\ref{escztery}
we have

\begin{corollary}
\label{T2.5}
Let $M$ and $N$ be compact connected oriented $4$-dimensional
Riemannian manifolds without
boundary, $3\leq p<4$. Then the degree is well defined,
integer valued and continuous in the space
$H^{1,p}(M,N)$ if and only if $N$ is not homeomorphic to $S^4$.
If $N$ is not homeomorphic to $S^4$ and $\pi_{3}(N)=0$,
then the degree is well defined, integer valued and continuous in $W^{1,p}(M,N)$.
\end{corollary}

The paper should be interesting mainly for people working in geometric analysis. Since
the proofs employ quite a lot of algebraic topology, we made some effort to
present our arguments from different points of view whenever it was possible.
For example, some proofs were presented both from the perspective of algebraic
topology and the perspective of differential forms.

The main results in the paper are Theorems~\ref{T3}, \ref{T4}, \ref{T1}, \ref{T2},
Corollaries~\ref{T3.5}, \ref{T2.5} and also
Theorem~\ref{T5}.

The paper is organized as follows.
In Section~\ref{RHS} we study basic examples and properties of rational homology spheres.
In Section~\ref{HD} we prove Theorems~\ref{T1} and~\ref{T2}.
The final Section~\ref{Orlicz}
provides a definition of the Orlicz-Sobolev space, its basic properties, and
the proofs of Theorems~\ref{T3}, and \ref{T4}.

\section{Rational homology spheres}
\label{RHS}

In what follows, $H^k(M)$ will stand for deRham cohomology groups. We say that a smooth
$n$-dimensional manifold $M$ without boundary is a
{\em rational homology sphere} if it has the same deRham cohomology as $S^n$, that is
$$
H^k(M)=\begin{cases}\bbbr & \text{ for }k=0 \text{ or }k=n,\\
0&\text{ otherwise}.
\end{cases}
$$
Clearly $M$ must be compact, connected and orientable.

If $M$ and $N$ are smooth compact connected oriented $n$-dimensional manifolds
without boundary, then the {\em degree} of a smooth mapping
$f:M\to N$ is defined by
$$
\deg f=
\frac{\int_M f^*\omega}{\int_N \omega}\, ,
$$
where $\omega$ is any $n$-form on $N$ with $\int_N\omega\neq 0$
($\omega$ is not exact by Stokes' theorem and hence it defines a non-trivial
element in $H^n(M)$).
It is well known that $\deg f\in\bbbz$, that it does not depend on the choice
of $\omega$ and that it
is a homotopy invariant.

The reason why rational homology spheres
play such an important role in the degree theory of Sobolev mappings
stems from the following result.
\begin{theorem}
\label{T5}
Let $M$ be a smooth compact connected oriented $n$-di\-men\-sion\-al manifold without boundary,
$n\geq 2$.
Then there is a smooth mapping $f:S^n\to M$ of nonzero degree if and only if the universal
cover of $M$ is a rational homology sphere.
\end{theorem}
\begin{proof}
Let us notice first  that if
a mapping $f:S^n\to M$ of nonzero degree exists,
then $M$ is a rational homology sphere.
Indeed, clearly \mbox{$H^0(M)=H^n(M)=\bbbr$}, because $M$ is compact, connected and oriented.
Suppose that $M$ has non-trivial cohomology in dimension $0<k<n$, that is -- there is a
closed $k$ form $\alpha$ that is not exact.
According to the Hodge Decomposition Theorem, \cite{warner},
we may also assume that $\alpha$ is
coclosed, so $*\alpha$ is closed.
Then, $\omega=\alpha\wedge *\alpha$ is an
$n$-form on $M$ such that
$$
\int_M \omega = \int_M|\alpha|^2 >0\, .
$$
We have
$$
\deg f =
\frac{\int_{S^n} f^*\omega}{\int_M\omega} =
\frac{\int_{S^n} f^*\alpha\wedge f^*(*\alpha)}{\int_M\omega} = 0
$$
which is a contradiction.
The last equality follows from the fact that $H^k(S^n)=0$, hence
the form $f^*\alpha$ is exact, $f^*\alpha =d\eta$. Since $f^*(*\alpha)$ is closed,
we have
$$
\int_{S^n} f^*\alpha\wedge f^*(*\alpha) = \int_{S^n} d(\eta\wedge f^*(*\alpha)) = 0
$$
by Stokes' theorem.

The contradiction proves that $H^k(M)=0$ for all $0<k<n$, so $M$ is a rational
homology sphere.

On the other hand, any such mapping $f$ factors through the uni\-vers\-al cover
$\widetilde{M}$ of $M$, because $S^n$ is simply connected,
see \cite[Proposition~1.33]{hatcher}.
$$
\xymatrix{
&\widetilde{M}\ar[d]^p\\
S^n \ar[ru]^{\tilde{f}} \ar[r]^f &M}
$$
Clearly, $\widetilde{M}$ is orientable and we can choose an orientation in such a way that
$p$ be orientation preserving.

Since $\deg f\neq 0$, $f^*:H^n(M)\to H^n(S^n)$ is an isomorphism.
The factorization gives
$f^*=\tilde{f}^*\circ p^*$, and hence
$\tilde{f}^*:H^n(\widetilde{M})\to H^n(S^n)=\bbbr$ is an isomorphism, so $\widetilde{M}$
is compact. This implies that if $\omega$ is a volume form on $\widetilde{M}$, then
$\int_{S^n}\tilde{f}^*\omega\neq 0$ and hence $\deg \tilde{f}\neq 0$.
Indeed, $\omega$ defines a generator (i.e. a non-zero element) in
$H^n(\widetilde{M})$ and hence $\tilde{f}^*\omega$ defines a generator in $H^n(S^n)$.
If $\eta$ is a volume form on $S^n$, then $\tilde{f}^*\omega=c\eta+d\lambda$, $c\neq 0$,
because $H^n(S^n)=\bbbr$ and the elements $\tilde{f}^*\omega$ and $\eta$ are proportional
in $H^n(S^n)$. Thus $\int_{S^n} \tilde{f}^*\omega=c\int_{S^n}\eta\neq 0$.

The last argument can be expressed differently. Since $\widetilde{M}$ is compact, the number
$\gamma$ of sheets of the covering is finite (it equals $|\pi_n(M)|$, see
\mbox{\cite[Proposition~1.32]{hatcher}}). It is easy to see that
$\deg f =\gamma\deg\tilde{f}$, so $\deg \tilde{f}\neq 0$.

Thus we proved that the mapping $\tilde{f}:S^n\to\widetilde{M}$ has nonzero degree and hence
$\widetilde{M}$ is a rational homology sphere, by the fact obtained at the beginning of
our proof.

We are now left with the proof that if $\widetilde{M}$ is a rational homology sphere,
then there is a mapping from $S^n$ to $M$ of nonzero degree. Clearly, it suffices to
prove that there is a mapping $f:S^n\to\widetilde{M}$ of nonzero degree, since the composition
with the covering map multiplies degree by the (finite) order $\gamma$ of the covering.

To this end we shall employ the Hurewicz theorem mod Serre class $\mathcal{C}$ of all finite
abelian groups (see \cite[Chapter X, Theorem 8.1]{Hu_HT}, also \cite{KlausKreck04},
\cite[Chapter 1, Theorem 1.8]{Hatcher_SS} and \cite[Ch.~9, Sec.~6, Theorem~15]{spanier})
that we state as a lemma.
\begin{lemma}
\label{hurewicz mod c}
Let $\mathcal{C}$ denote the class of all finite abelian groups. If $X$ is a simply
connected space and $n\geq 2$ is an integer such that $\pi_m(X)\in \mathcal{C}$
whenever $1<m<n$, then the natural Hurewicz homomorphism
$$
h_m:\pi_m(X) \to H_m(X,\bbbz)
$$
is a $\mathcal{C}$-isomorphism (i.e. both $\ker h_m$ and $\coker h_m$ lie in $\mathcal{C}$)
whenever \mbox{$0<m\leq n$.}
\end{lemma}
Since $\widetilde{M}$ is a simply connected rational homology sphere, the integral
homology groups in dimensions $2,3,\ldots,n-1$ are finite abelian groups: integral
homology groups of compact manifolds are finitely generated, thus of
form $\bbbz^\ell\oplus \bbbz_{p_1}^{k_1}\oplus\cdots\oplus \bbbz_{p_r}^{k_r}$,
and if we calculate $H^m(\widetilde{M})=H^m(\widetilde{M},\bbbr)$ with the help
of the Universal Coefficient Theorem (\cite[Theorem 3.2]{hatcher}), we
have $0=H^m(\widetilde{M})=\Hom(H_m(\widetilde{M},\bbbz),\bbbr)$.
However, $\Hom(\bbbz^\ell,\bbbr)=\bbbr^\ell$, thus there can be no  free abelian
summand $\bbbz^\ell$ in $H_m(\widetilde{M},\bbbz)$ -- and all that is possibly left is
a finite abelian group.

Now we proceed by induction to show that the hypotheses of Lemma~\ref{hurewicz mod c}
are satisfied.
We have $\pi_1(\widetilde{M})=0$, thus
$h_2:\pi_2(\widetilde{M}) \to H_2(\widetilde{M},\bbbz)$ is a $\mathcal{C}$-iso\-morphism.
Since $H_2(\widetilde{M},\bbbz)$ is, as we have shown, in $\mathcal{C}$,
the group $\pi_2(\widetilde{M})$ is a finite abelian group as well, and we can apply
Theorem~\ref{hurewicz mod c} to show that $h_3$ is a $\mathcal{C}$-isomorphism and
thus that $\pi_3(\widetilde{M})\in\mathcal{C}$.

We proceed likewise by induction until we have that $h_n$ is a $\mathcal{C}$-isomorphism
between $\pi_n(\widetilde{M})$ and $H_n(\widetilde{M},\bbbz)=\bbbz$.
Since $\coker h_n=\bbbz/h_n(\pi_n(\widetilde{M}))$ is a finite group, there exists a
non-zero element $k$ in the image of $h_n$, and $[f]\in\pi_n(\widetilde{M})$ such
that $h_n([f])=k\neq 0$.

The generator of  $H_n(\widetilde{M},\bbbz)=\bbbz$ is the cycle class given by the whole
manifold $\widetilde{M}$, i.e. $h_n([f])=k[\widetilde{M}]$; in other words -- having
fixed a volume form $\omega$ on $\widetilde{M}$, we identify a cycle class $[C]$
with an integer by
$$
[C]\longmapsto \left(\int_C\omega\right)\bigg/\left(\int_{\widetilde{M}}\omega\right),
$$
We also recall that the Hurewicz homomorphism $h_n$ attributes to
an \mbox{$[f]\in \pi_n(\widetilde{M})$}  a cycle class
$h_n([f])=f_*[S^n]$, where $[S^n]$
is the cycle that generates $H_n(S^n,\bbbz)=\bbbz$. Altogether,
$$
\deg f=
\frac{\int_{S^n}f^*\omega }{\int_{\widetilde{M}}\omega} =
\frac{\int_{f_*[S^n]}\omega}{\int_{\widetilde{M}}\omega} =
k\neq 0,
$$
since $f_*[S^n]=k[\widetilde{M}]$.
\end{proof}

To see the scope of applications of Theorem~\ref{T3} and~\ref{T2}
it is important to understand
how large the class of rational homology spheres is, or, more precisely, the class
of manifolds whose universal cover is a rational homology sphere.

One can conclude from the proof presented above
that if $\widetilde{M}$ is a rational
homology sphere then $M$ is a rational homology sphere as well. We will provide two
different proofs of this fact, the second one being related to the argument presented above.

\begin{theorem}
\label{T6}
Let $M$ be an $n$-dimensional compact orientable manifold
without boundary such that its
universal cover $\widetilde{M}$ is a rational homology $n$-sphere. Then
\begin{itemize}
\item[a)] $M$ is a rational homology $n$-sphere as well.
\item[b)] If $n$ is even, then $M$ is simply connected and
$\widetilde{M}=M$.
\end{itemize}
\end{theorem}
\begin{proof}
We start with two trivial observations: the fact that $\widetilde{M}$ is
a rational homology sphere implies immediately that $\widetilde{M}$ is compact
($H^n(\widetilde{M})=\bbbr$) and connected ($H^0(\widetilde{M})=\bbbr$). This, in
turn, shows that $M$ is connected and that the number of sheets in
the covering is finite.

For any finite covering $p:\widetilde{N}\to N$ with
number of sheets $\gamma$ there is not only the induced {\em lifting}
homomorphism $p^*:H^k(N)\to H^k(\widetilde{N})$, but also the so-called
{\em transfer} or {\em pushforward} homomorphism
$$
\tau_*:H^k(\widetilde{N})\to H^k(N)
$$
defined on differential forms as follows. For any $x\in N$ there
is a neighborhood $U$ such that $p^{-1}(U)=U_1\cup\ldots\cup
U_\gamma$ is a disjoint sum of open sets such that
$$
p_i=p|_{U_i}:U_i\to U
$$
is a diffeomorphism. Then for a $k$-form $\omega$ on $\widetilde{N}$ we define
$$
\tau_*\omega|_U =
\sum_{i=1}^\gamma (p_i^{-1})^*(\omega|_{U_i})\, .
$$
If $U$ and $W$ are two different neighborhoods in $N$, then
$\tau_*\omega|_U$ coincides with $\tau_*\omega|_W$ on $U\cap W$ and hence $\tau_*\omega$
is globally defined on $N$.
Since
$$
\tau_*(d\omega)|_U=
\sum_{i=1}^\gamma(p_i^{-1})^*(d\omega|_{U_i})=
\sum_{i=1}^\gamma d(p_i^{-1})^*(\omega|_{U_i}) =
d(\tau_*\omega|_U)
$$
we see that $\tau_*\circ d = d\circ\tau_*$ and hence $\tau_*$ is a homomorphism
$$
\tau_*:H^k(\widetilde{N})\to H^k(N)\, .
$$
Clearly $\tau_*\circ p^*=\gamma\,{\rm id}$ on $H^k(N)$ and hence
$\tau_*$ is a surjection.

In our case $\widetilde{N}=\widetilde{M}$ is a rational homology sphere and
all cohomology groups $H^k(\widetilde{M})$ vanish for $0<k<n$.
Since $\tau_*$ is a surjection, $H^k(M)=0$. The remaining $H^0(M)$ and $H^n(M)$
are equal $\bbbr$ by compactness, orientability and connectedness of $M$.

Another proof is related to the argument used in the proof of Theorem~\ref{T5}.
By contradiction, suppose that $H^k(M)\neq 0$ for some $0<k<n$.
Hence there is a non-trivial closed and coclosed $k$-form $\alpha$ on $M$.
Since $\alpha$ is coclosed, $*\alpha$ is closed.
Then
$$
\int_M\alpha \wedge \ast\alpha
=
\int_M |\alpha|^2 > 0,
$$
thus
$$
0\neq \gamma\int_M \alpha\wedge *\alpha =
\int_{\widetilde{M}}p^*(\alpha\wedge
\ast\alpha)=\int_{\widetilde{M}}p^*\alpha\wedge p^*(\ast\alpha)=0,
$$
because $p^*\alpha$ is exact on $\widetilde{M}$
(we have $H^k(\widetilde{M})=0$), $p^*(*\alpha)$ is closed, and hence $p^*\alpha\wedge
p^*(\ast\alpha)$ is exact.

To prove b), we recall that if $p:\widetilde{N} \to N$ is a covering of
order $\gamma$, then the Euler characteristic $\chi_{\widetilde{N}}$ of
$\widetilde{N}$ is equal to $\gamma \chi_N$. This can be easily seen
through the triangulation definition of $\chi_N$: to any sufficiently small simplex in
$N$ there correspond $\gamma$ distinct simplices in $\widetilde{N}$ -- in
such a way we may lift a sufficiently fine triangulation of $N$ to $\widetilde{N}$;
clearly, the alternating sum of number of simplices in every
dimension calculated for $\widetilde{N}$ is $\gamma$ times that for $N$.

In our case, the Euler characteristic of $M$ and $\widetilde{M}$ is
either 0 (for $n$ odd) or 2 (for $n$ even) -- this follows
immediately from the fact that
$$
\chi_{N}=\sum_{i=0}^{\dim N}
(-1)^i\dim H^i(N).
$$
Therefore, for $n=\dim M$ even, the order of
the universal covering $\widetilde{M}\to M$ must be 1, and b) follows.
\end{proof}

\begin{remark}
We actually proved a stronger version of part a).
If a cover (not necessarily universal) of a
manifold $M$ is a rational homology sphere, then $M$ is a rational homology
sphere, too -- in the proof we never used the fact that the covering space is
simply connected. However, we stated the result for the universal cover only, because
this is exactly what we need in our applications.
\end{remark}

Below we provide some examples of rational homology spheres.

Any homology sphere is a rational homology sphere.
That includes Poincar\'e homology sphere (also known as Poincar\'e
dodecahedral space) and Brieskorn manifolds
$$
\Sigma(p,q,r) =
\left\{ (z_1,z_2,z_3)\in\bbbc^3:\,
z_1^p+z_2^q+z_3^r=0,\
|(z_1,z_2,z_3)|=1\right\}\, ,
$$
where $1<p<q<r$ are pairwise relatively prime integers.
$\Sigma(2,3,5)$ is the Poincar\'e homology sphere.

Recall that any isometry of $\bbbr^n$ can be described
as a composition of an orthogonal linear map
and a translation; the group of all such isometries, denoted by
$E(n)$, is a semi-direct product $O(n) \ltimes \bbbr^n$. Any discrete
subgroup $G$ of $E(n)$ such that $E(n)/G$ is compact is called a
{\it crystallographic group}.

\begin{proposition}[see \cite{Szczepanski83}]
Let us set $\{e_i\}$ to be the standard basis in
$\bbbr^{2n+1}$ and
$$B_i=\mathrm{Diag}(-1,\cdots,-1,\underbrace{1}_i,-1,\cdots,-1).$$ Let
$\Gamma$ be a crystallographic group generated by isometries of
$\bbbr^{2n+1}$
\begin{equation*}
T_i\,:\,v\longmapsto B_i\cdot v+e_i,\qquad i=1,2,\ldots,2n.
\end{equation*}
Then $M=\bbbr^{2n+1}/\Gamma$ is a rational homology $(2n+1)$-sphere.
\end{proposition}
For $n=1$ this is a well known example
of $\bbbr^3/\mathcal{G}_6$, see \cite[3.5.10]{Wolf_CC}. However,
$\pi_1(M)=\Gamma$ is infinite and the universal cover of $M$ is
$\bbbr^{2n+1}$, so $M$ is an example of a rational homology sphere
whose universal cover is not a rational homology sphere.

The following well known fact illustrates the difficulty of finding non-trivial
examples in dimension 4:
\begin{proposition}
\label{escztery}
Let $N$ be a compact orientable $4$-manifold without boundary.
Then the universal cover of $N$ is a rational homology sphere if
and only if $N$ is homeomorphic to $S^4$.
\end{proposition}
\begin{proof}
If $N$ is homeomorphic to $S^4$, then $\widetilde{N}=N$ is a rational
homology sphere. Suppose now
that the universal cover $\widetilde{N}$ is a rational homology sphere.
Theorem~\ref{T6} tells
us that $N$ must be simply connected -- therefore $H_1(N,\bbbz)=0$.
Standard application of the Universal Coefficients Theorem (see e.g.
\cite[Corollary 15.14.1]{BottTu}) gives us that $H^2(N,\bbbz)$ has no
torsion component (and thus is 0, since $N$ is a rational homology
sphere) and $H^1(N,\bbbz)=0$. Poincar\'e duality, in turn, shows that
$H^3(N,\bbbz)\approx H_1(N,\bbbz)=0$. The remaining $H^0(N,\bbbz)$ and
$H^4(N,\bbbz)$ are $\bbbz$ by connectedness and orientability of $N$.
Altogether, $N$ is an integral homology sphere, and thus, by
the homology Whitehead theorem (\cite[Corollary 4.33]{hatcher}), a homotopy sphere.
Ultimately, by M. Freedman's celebrated result on Generalized
Poincar\'e Conjecture (\cite{Freedman}) a homotopy 4-sphere is
homeomorphic to a $S^4$. Whether $M$ is diffeomorphic to $S^4$ remains,
however, an open problem.
\end{proof}

One can construct lots of examples of non-simply connected rational
homology $4$-spheres, although not every finite group might arise as
a fundamental group of such a manifold
(\cite[Corollary 4.4]{HambletonKreck}, see also \cite{Kervaire69}).

As for higher even dimensions, A. Borel proved (\cite{Borel49}) that the
only rational homology $2n$-sphere that is a homogeneous $G$-space
for some compact, connected Lie group $G$ is a standard $2n$-sphere.

Another example is provided by lens spaces
\cite[Example~2.43]{hatcher}. Given an integer $m>1$ and integers
$\ell_1,\ldots,\ell_n$ relatively prime to $m$, the
{\em lens space}, $L_m(\ell_1,\ldots\ell_n)$ is defined as the orbit
space $S^{2n-1}/\bbbz_m$ of $S^{2n-1}\subset\bbbc^n$
with the action of $\bbbz_m$ generated by rotations
$$
\rho(z_1,\ldots,z_n) =
\left(e^{2\pi i\ell_1/m}z_1,\ldots,e^{2\pi i\ell_n/m} z_n\right)\, .
$$
Since the action of $\bbbz_m$ on $S^{2n-1}$ is free, the projection
$$S^{2n-1}\to L_m(\ell_1,\ldots,\ell_n)$$ is a covering and hence
lens spaces are manifolds. One can easily prove that the integral
homology groups are
$$
H_k\left(L_m(\ell_1,\ldots,\ell_n)\right) =
\left\{
\begin{array}{ccc}
\bbbz    & \mbox{if $k=0$ or $k=2n-1$,}\\
\bbbz_m  & \mbox{if $k$ is odd, $0<k<2n-1$,} \\
0        & \mbox{otherwise.}
\end{array}
\right.
$$
Hence $L_m(\ell_1,\ldots,\ell_n)$ is a rational homology sphere with the
universal covering space $S^{2n-1}$.

In dimensions $n>4$ there exist numerous simply connected (smooth)
rational homology spheres -- a particularly interesting set of
examples are the exotic spheres, i.e. manifolds homeomorphic, but
not diffeomorphic to a sphere (see e.g. a nice survey article of
Joachim and Wraith \cite{JoachimWraith}).

\section{Proofs of Theorems~\ref{T1} and~\ref{T2}.}
\label{HD}

In the proof of these two theorems we shall need the following definition and a
theorem of B. White:

\begin{definition}
Two continuous mappings $g_1, g_2:M\to N$ are {\em $\ell$-homotopic}
if there exists a
triangulation of $M$ such that the two mappings restricted to the $\ell$-dimensional skeleton
$M^\ell$ are homotopic in $N$.
\end{definition}

It easily follows from the cellular approximation theorem that this definition
does not depend on the choice of a triangulation, see \cite[Lemma~2.1]{hangl2}.
\begin{lemma}[Theorem~2, \cite{white1}]
\label{thm white}
Let $M$ and $N$ be compact Riemannian manifolds, $f\in W^{1,p}(M,N)$. There exists $\eps>0$
such that any two Lipschitz mappings $g_1$ and $g_2$ satisfying $\|f-g_i\|_{W^{1,p}}<\eps$
are $[p]$-homotopic.
\end{lemma}
Here $[p]$ is the largest integer less than or equal to $p$.
We shall also need the following construction:
\begin{example}
\label{ex 3.3}
We shall construct a particular sequence of mappings \mbox{$g_k:S^n\to S^n$}.
Consider the spherical
coordinates on $S^n$:
\begin{equation*}
\begin{split}
(z,\theta)\mapsto &(z\sin\theta,\cos \theta),\\
&z\in S^{n-1}
\text{ (equatorial coordinate)},\\
&\theta\in[0,\pi] \text{ (latitude angle)}.
\end{split}
\end{equation*}
These coordinates have, clearly, singularities at the north and south pole.
Consider the polar cap $C_k=\{(z,\theta)~~:~~0\leq \theta\leq \frac{1}{k}\}$ and a mapping
$$
g_k\,:\,S^n\to S^n,\quad g_k(z,\theta)=
\begin{cases}
(z,k\pi \theta)&0\leq\theta< \frac{1}{k}\\
(z,\pi) \text{ (i.e. south pole)}& \frac{1}{k}\leq \theta\leq \pi.
\end{cases}.
$$
The mapping stretches the polar cap $C_k$ onto the whole sphere, and maps all the rest of
the sphere into the south pole.
It is clearly homotopic to the identity map -- therefore it is of degree 1.

Measure of the polar cap $C_k$ is comparable to $k^{-n}$,
$|C_k|\approx k^{-n}$.
It is also easy to see that the derivative of $g_k$ is bounded by $Ck$.

Let $1\leq p<n$. Since the mappings $g_k$ are bounded and $g_k$ converges a.e. to the
constant mapping into the south pole as $k\to\infty$, we conclude that $g_k$ converges to a
constant map in $L^p$.
On the other hand, the $L^p$ norm of the derivative $Dg_k$ is bounded by
$$
\int_{S^n} |Dg_k|^p \leq Ck^p |C_k| \leq
C'k^{p-n}\to 0
\quad
\mbox{as $k\to\infty$,}
$$
and hence $g_k$ converges to a constant map in $W^{1,p}$.
\end{example}

\begin{proof}[Proof of Theorem~\ref{T1}]
Assume that $\pi_n(N)=0$ and
$f\in H^{1,p}(M,N)$, with $n-1\leq p<n$. By Lemma~\ref{thm white} any two smooth mappings
$g_1$ and $g_2$ sufficiently close to $f$ are $(n-1)$-homotopic, that is there exists a
triangulation $\mathfrak{T}$ of $M$ such that $g_1$ and $g_2$, restricted to the
$(n-1)$-dimensional skeleton $M^{n-1}$ of $\mathfrak{T}$, are homotopic.

Let $H:M^{n-1}\times [0,1] \to N$ be a homotopy between $g_1$ and $g_2$ and let $\Delta$
be an arbitrary $n$-simplex of $\mathfrak{T}$. We have defined a mapping $\mathcal{H}$
from the boundary of $\Delta\times [0,1]$ to $N$: it is given by $g_1$ on
$\Delta\times \{0\}$, by $g_2$ on $\Delta \times \{1\}$ and by $H$ on
$\partial \Delta \times [0,1]$. However, $\partial(\Delta\times [0,1])$ is homeomorphic to
$S^{n}$. Since $\pi_{n}(N)=0$, any such mapping is null-homotopic and extends to the whole
$\Delta\times [0,1]$. In such a way, simplex by simplex, we can extend the homotopy
between $g_1$ and $g_2$ onto the whole $M$.

We showed that any two smooth mappings sufficiently close to
$f$ in the norm of $W^{1,p}$ are homotopic,
and we may define the homotopy class of $f$ as the homotopy class of a sufficiently good
approximation of $f$ by a smooth function.
Hence homotopy classes can be well defined in $H^{1,p}(M,N)$.

If in addition $\pi_{n-1}(N)=0$, then according to
\cite[Corollary~1.7]{hangl2}, smooth mappings
are dense in $W^{1,p}(M,N)$, so
$W^{1,p}(M,N)=H^{1,p}(M,N)$ and thus homotopy classes
are well defined in $W^{1,p}(M,N)$.

In order to complete the proof of the theorem we need yet to show that
if $\pi_n(N)\neq 0$, we cannot define homotopy classes in
$W^{1,p}(S^n,M)$, \mbox{$n-1\leq p<n$}, in a continuous way. As advertised in the
Introduction, we shall construct a sequence of smooth mappings that
converge in $W^{1,p}(S^n,N)$ to a constant mapping, but are homotopically non-trivial.

If $\pi_n(N)\neq 0$, we have a smooth
mapping $G:S^n\to N$ that is not homotopic to a constant one.
Hence the mappings $G_k=G\circ g_k$ are not homotopic to
a constant mapping, where $g_k$ was constructed in Example~\ref{ex 3.3}.
On the other hand, since the mappings $g_k$ converge to a constant map
in $W^{1,p}$, the sequence $G_k$ also converges to a constant map, because
composition with $G$ is continuous in the Sobolev norm.
\end{proof}

The condition $\pi_n(N)=0$ is not necessary for the mappings
$g_1$ and $g_2$ to be homotopic, as can be seen from the following well known
\begin{proposition}
\label{cp2tocp1}
Any two continuous mappings $\mathbb{CP}^2\to \mathbb{CP}^1$ are homotopic,
while $\pi_4(\mathbb{CP}^1)=\bbbz_2$.
\end{proposition}
\begin{proof}[Sketch of a proof of Proposition~\ref{cp2tocp1}]
Since $\mathbb{CP}^1=S^2$, we have
$\pi_4(\mathbb{CP}^1)=\pi_4(S^2)=\bbbz_2$
(see \cite[p. 339]{hatcher}).

The space $\mathbb{CP}^2$ can be envisioned as $\mathbb{CP}^1=S^2$
with a 4-disk $D$ attached along its boundary by the Hopf mapping $H:S^3\to S^2$.
Therefore the 2-skeleton $(\mathbb{CP}^2)^{(2)}$ consists of the sphere $S^2$.

Suppose now that we have a mapping
$\phi:\mathbb{CP}^2\to \mathbb{CP}^1$. If $\phi$ is not null-homotopic
(i.e. homotopic to a constant mapping) on $(\mathbb{CP}^2)^{(2)}$,
then its composition with the Hopf mapping is not null-homotopic either,
since $H$ generates $\pi_3(S^2)=\bbbz$. Then $\phi$ restricted to the
boundary $S^3$ of the 4-disk $D$ is not null-homotopic and cannot be
extended onto $D$, and thus onto the whole $\mathbb{CP}^2$. Therefore we
know that $\phi$ restricted to the 2-skeleton of $\mathbb{CP}^2$ is null-homotopic.

This shows that $\phi$ is homotopic to a composition
$$
\mathbb{CP}^2\xrightarrow{p}\mathbb{CP}^2/(\mathbb{CP}^2)^{(2)}=S^4\to \mathbb{CP}^1=S^2.
$$
It is well known that $\pi_4(S^2)=\mathbb{Z}_2$ and that the only non-null-homotopic
mapping $S^4\to S^2$ is obtained by a composition of the Hopf mapping $H:S^3\to S^2$
and of its suspension $\Sigma H:S^4\to S^3$ (see \cite{hatcher}, Corollary~4J.4
and further remarks on EHP sequence). Then, if $\phi$ is to be non-null-homotopic,
it must be homotopic to a composition
$$
\mathbb{CP}^2\xrightarrow{p}S^4\xrightarrow{\Sigma H}S^3\xrightarrow{H}S^2.
$$
The first three elements of this sequence are, however, a part of the cofibration sequence
$$
S^2\hookrightarrow
\mathbb{CP}^2\xrightarrow{p}S^4\xrightarrow{\Sigma H}S^3\rightarrow\Sigma
\mathbb{CP}^2\rightarrow\cdots,
$$
and as such, they are homotopy equivalent to the sequence
$$
S^2\hookrightarrow \mathbb{CP}^2\rightarrow \mathbb{CP}^2\cup C(S^2)
\hookrightarrow \mathbb{CP}^2\cup C(S^2) \cup C(\mathbb{CP}^2)\approx
\mathbb{CP}^2\cup C(S^2)/\mathbb{CP}^2,
$$
since attaching a cone to a subset is homotopy equivalent to contracting
this subset to a point (by $C(A)$ we denote a cone of base $A$) . One can
clearly see that in the above sequence the composition of the second and
third mapping and the homotopy equivalence at the end,
$\mathbb{CP}^2\rightarrow \mathbb{CP}^2\cup C(S^2)/\mathbb{CP}^2$,
is a constant map; thus the original map $\Sigma H\circ p:\mathbb{CP}^2\to S^3$
is homotopic to a constant map, and so is
$H\circ\Sigma H\circ p:\mathbb{CP}^2\to S^2$, the only candidate for
a non-null-homotopic mapping between these spaces.
\end{proof}

\begin{proof}[Proof of Theorem~\ref{T2}]
This proof consist of two parts:

In the first one, suppose $N$ is such that its universal cover $\widetilde{N}$ is a rational
homology sphere; $n=\dim M=\dim N$. We shall construct an explicit example of a sequence of
mappings in $H^{1,p}(M,N)$, $n-1\leq p<n$
of a fixed, non-zero degree, that converge in $H^{1,p}$-norm
to a constant mapping.

The manifold $M$ can be smoothly mapped onto an $n$-dimensional sphere
(take a small open ball in $M$ and map its complement into a the south pole).
This mapping is
clearly of degree 1 -- we shall denote it by $F$. Notice that $F$, by
construction, is a diffeomorphism between an open set in $M$
and $S^n\setminus\{\text{south pole}\}$.
Next, let us consider a smooth mapping $G:S^n\to N$ of non-zero degree, the existence of
which we have asserted in Theorem~\ref{T5}.

We define a sequence of mappings $F_k:M\to N$ as a composition of $F$,
mappings $g_k$ given by Example~\ref{ex 3.3} and $G$:
\begin{equation}
\label{def F_k}
F_k=G\circ g_k\circ F.
\end{equation}

The degree of $F_k$ is equal to $\deg G$, thus non-zero
and constant. On the other hand, we can, exactly as in the proof of Theorem~\ref{T1},
prove that $F_k$ converges in $W^{1,p}$ to a constant map.
Since $F_k$ converges a.e. to a constant map, it converges in $L^p$.
As concerns the derivative,
observe that $|DF_k|\leq Ck$ and $DF_k$ equals zero outside the set
$F^{-1}(C_k)$ whose measure is comparable to $k^{-n}$. Hence
$$
\int_M |DF_k|^p\leq
C k^p\, |F^{-1}(C_k)|\leq
C'k^{p-n}\to 0
\quad
\mbox{as $k\to\infty$.}
$$
and thus the convergence to the constant map is in $W^{1,p}$.

This shows that the mappings $F_k$ can be arbitrarily
close to the constant map, so the degree cannot
be defined as a continuous function on $H^{1,p}(M,N)$, $n-1\leq p<n$.

Now the second part: we shall prove that, as long as the universal cover
$\widetilde{N}$ of $N$
is not a rational homology sphere, the degree is well defined in $H^{1,p}(M,N)$,
$n-1\leq p<n$.
To this end, we will show that it is continuous on $C^\infty(M,N)$ equipped with the
$W^{1,p}$ distance and therefore it extends continuously onto the whole $H^{1,p}(M,N)$.
In fact, since the degree mapping takes value in $\bbbz$, we need to show that any two smooth
mappings that are sufficiently close in $H^{1,p}(M,N)$ have the same degree.

Consider two smooth mappings $g,h:M\to N$  that
are sufficiently close to a given mapping in $H^{1,p}(M,N)$.
By Lemma~\ref{thm white} of White, $g$ and $h$ are ($n-1$)-homotopic, i.e. there exists a
triangulation $\mathfrak{T}$ of $M$ such that $g$ and $h$ are homotopic on the ($n-1$)-skeleton $M^{n-1}$. Let $H\,:\,M^{n-1}\times[0,1]\to N$ be the homotopy, $H(x,0)=g(x)$,
$H(x,1)=h(x)$ for any $x\in M^{n-1}$.

Let us now look more precisely at the situation over a fixed $n$-simplex $\Delta$ of the
triangulation $\mathfrak{T}$. We have just specified a continuous mapping $\mathcal{H}$ on
the boundary of $\Delta\times[0,1]$: it is given by $g$ on $\Delta\times\{0\}$,
by $h$ on $\Delta\times\{1\}$ and by $H$ on $\partial \Delta\times[0,1]$.

Note that $\partial (\Delta\times[0,1])$ is homeomorphic to $S^n$, therefore the mapping
$\mathcal{H}:\partial (\Delta\times[0,1])\to N$ is of degree zero, since $\widetilde{N}$ is not
a rational homology sphere. We thus may fix orientation of $\partial\Delta\times [0,1]$,
so that
\begin{equation}
\label{eq:deg 0 over Delta}
\int_\Delta J_g=\int_\Delta J_h +\int_{\partial\Delta\times [0,1]} J_H.
\end{equation}
Recall that the Jacobian $J_f$ of a
function $f:M\to N$, with fixed volume forms $\mu$ on $M$ and $\nu$ on $N$, is
given by the relation $f^*\nu=J_f\mu$, and
$\deg f=(\int_M J_f \,d\mu)/(\int_N d\nu)=|N|^{-1} \int_M J_f\, d\mu$.
Therefore, by summing up the relations \eqref{eq:deg 0 over Delta} over all
the $n$-simplices
of $\mathfrak{T}$ we obtain
\begin{equation*}
\begin{split}
|N| \deg g &=\int_M J_g \,d\mu\\&=\sum_{\Delta^n \in \mathfrak{T}}
\left(\int_{\Delta^n} J_h \,d\mu+\int_{\partial \Delta^n\times[0,1]} J_H\,d\mu\right)\\
&=\int_M J_h\,d\mu+\sum_{\Delta^n \in \mathfrak{T}}\int_{\partial \Delta^n\times[0,1]}
J_H\,d\mu\\&=|N| \deg h+\sum_{\Delta^n \in \mathfrak{T}}\int_{\partial \Delta^n\times[0,1]} J_H\, d\mu.
\end{split}
\end{equation*}
We observe that every face of $\partial \Delta^n\times[0,1]$ appears in the above
calculation twice, and with opposite orientation, thus
$\sum_{\Delta^n \in \mathfrak{T}}\int_{\partial \Delta^n\times[0,1]} J_H$
cancels to zero, and $\deg g=\deg h$.
Finally, if the universal cover of $N$ is not a rational homology sphere
and $\pi_{n-1}(N)=0$, then $H^{1,p}(M,N)=W^{1,p}(M,N)$ by
\cite[Corollary~1.7]{hangl2} and hence the degree is well defined in $W^{1,p}(M,N)$.
\end{proof}

\section{Orlicz-Sobolev spaces and proofs of Theorems~\ref{T3} and~\ref{T4}}
\label{Orlicz}
We shall begin by recalling some basic definitions of Orlicz and Orlicz-Sobolev spaces;
for a more detailed treatment see e.g. \cite[Chapter 8]{AdamsFournier}
and \cite[Chapter 4]{HIMO}.

Suppose $P:[0,\infty) \to [0,\infty)$ is convex, strictly increasing, with $P(0)=0$.
We shall call a function satisfying these conditions a \emph{Young function}.
Since we want to deal with Orlicz spaces that are very close to $L^n$, we will
also assume that $P$ satisfies the
so-called \emph{doubling} or \emph{$\Delta_2$-condition}:
\begin{equation}
\label{doubling cond}
\text{there exists }K>0\text{ such that }P(2t)\leq K\,P(t)\text{ for all }t\geq 0.
\end{equation}
This condition is very natural in our situation and it simplifies
the theory a great deal. Under the doubling condition the {\em Orlicz space}
$L^P(X)$ on a measure space $(X,\mu)$ is defined as a class of all measurable
functions such that $\int_X P(|f|)\, d\mu<\infty$. It is a Banach space with respect
to the so-called {\em Luxemburg norm}
$$
\|f\|_P=\inf\left\{k>0~:~\int_X P(|f|/k)\leq 1\right\}\, .
$$

We say that a sequence $(f_k)$ of functions in $L^P(X)$ converges to $f$
\emph{in mean}, if
$$
\lim_{k\to\infty}\int_X P(|f_k-f|)=0.
$$
It is an easy exercise to show that under the doubling condition
the convergence in $L^P$ is equivalent to the convergence
in mean.

For an open set $\Omega\subset\bbbr^n$
we define the {\em Orlicz-Sobolev} space $W^{1,P}(\Omega)$ as the space of all
the weakly differentiable functions on $\Omega$ for which the norm
$$
\|f\|_{1,P}=\|f\|_{L^1}+\sum_{i=1}^m \|D_i f\|_P
$$
is finite. For example, if $P(t)=t^p$, then $W^{1,P}=W^{1,p}$.
In the general case, convexity of $P$ implies that $P$ has at least linear growth and hence
$L^P(\Omega)\subset L^1_{\rm loc}(\Omega)$, $W^{1,P}(\Omega)\subset W^{1,1}_{\rm loc}(\Omega)$.
Using coordinate maps one can then easily extend the definition of the
Orlicz-Sobolev space to compact Riemannian manifolds, with the resulting space denoted by $W^{1,P}(M)$.

Let now $M$ and $N$ be compact Riemannian manifolds, $n=\dim M$.
We shall be interested in Orlicz-Sobolev spaces that are small
enough to exclude, for $M=B^n$, $N=S^{n-1}$, the radial projection
$x\mapsto x/|x|$. As shown in \mbox{\cite[p. 2]{HIMO}}, to exclude
such projection it is necessary and sufficient for $P$ to grow
fast enough to satisfy the so-called \emph{divergence condition},
already announced (see \eqref{E1}) in the Introduction:
\begin{equation*}
\int_1^\infty \frac{P(t)}{t^{n+1}}=\infty.
\end{equation*}
The function $P(t)=t^n$ satisfies this condition, but
we are not interested in Orlicz-Sobolev spaces that are to
small -- that are contained in $W^{1,n}(M,N)$, so
we impose an additional growth condition on $P$,
already stated (see \eqref{E1.5}) in the Introduction:
$$
P(t)=o(t^n)
\quad
\mbox{as $t\to\infty$.}
$$
This condition is important. In the case of $P(t)=t^n$ the degree
and homotopy results for mappings between manifolds are well known and
in this instance the results hold without any topological assumptions
about the target manifold $N$. Our aim is to extend the results beyond
the class $W^{1,n}$ and the dependence on the topological structure
of $N$ is revealed only when the Orlicz-Sobolev space is larger than
$W^{1,n}$, so we really need this condition.

In order to have $C^\infty(M,N)$ functions dense in $L^P(M,N)$, we need yet
another technical assumption: that the function $P$ does not `slow down' too much,
more precisely, that
\begin{equation}
\label{growth cond 2}
\text{the function } t^{-\alpha} P(t) \text{ is non-decreasing for some }\alpha>n-1.
\end{equation}
This condition is also natural for us. We are interested in the
Orlicz-Sobolev spaces that are just slightly larger than
$W^{1,n}$, so we are mainly interested in the situation when the
growth of $P$ is close to that of $t^n$ and the above condition
requires less than that. Note that it implies
$$
W^{1,P}(M,N)\hookrightarrow W^{1,\alpha}(M,N)\hookrightarrow W^{1,n-1}(M,N).
$$

The condition \eqref{growth cond 2} plays an important role in the
proof of density of smooth mappings.
\begin{lemma}\cite[Theorem~5.2]{HIMO}
\label{T7}
If the Young function $P(t)$ satisfies the divergence condition \eqref{E1},
doubling condition \eqref{doubling cond} and growth condition \eqref{growth cond 2},
then $C^{\infty}(M,N)$ mappings are dense in $W^{1,P}(M,N)$.
\end{lemma}
In particular, in order to extend continuously the notions of degree and
homotopy classes to $W^{1,P}(M,N)$, it is enough to prove that they are continuous
on $C^{\infty}(M,N)$ endowed with $W^{1,P}$ norm (provided $P$ satisfies all the
hypotheses of Lemma~\ref{T7}).

Note that density of smooth mappings in $W^{1,P}(M,N)$ and the embedding
$W^{1,P}(M,N)\hookrightarrow W^{1,n-1}(M,N)$ implies that
$W^{1,P}(M,N)\hookrightarrow H^{1,n-1}(M,N)$

\begin{proof}[Proof of Theorem~\ref{T4}]
Let $M$, $N$ and $P$ be as in the statement of Theorem~\ref{T4}. Assume also that
$\pi_n(N)=0$.
By the above remark, it is enough to prove that if two smooth mappings $f,g:M\to N$
are sufficiently close to a $W^{1,P}(M,N)$ mapping in $\|\cdot\|_{1,P}$-norm
then they are homotopic.
However, by the inclusion $W^{1,P}(N,M)\hookrightarrow H^{1,n-1}(N,M)$,
we know that $f$ and $g$ are close in $H^{1,n-1}(N,M)$, and by Theorem~\ref{T1}
they are homotopic.

If $\pi_n(N)\neq 0$, we can construct, like in the proof of Theorem~\ref{T1},
a sequence of non-nullhomotopic mappings convergent to a constant mapping in $W^{1,P}(S^n,N)$.
Indeed, the mappings $G_k$ constructed as in the proof of Theorem~\ref{T1}
converge a.e. to a constant mapping, so they converge to a constant mapping in mean
and hence in $L^P$. On the other hand, the derivative of $G_k=G\circ g_k$ is bounded
by $Ck$ and hence
$$
\int_{S^n} P(|DG_k|)\leq P(Ck)|C_k|\leq C'P(k)k^{-n} \to 0
\quad
\mbox{as $k\to\infty$}
$$
by \eqref{E1.5} and the doubling condition. Hence $DG_k$ converges to zero in mean
and thus in $L^P$.
\end{proof}

\begin{proof}[Proof of Theorem~\ref{T3}]
If the universal cover of $N$ is not a rational homology sphere,
the continuity of the degree mapping with respect to the Orlicz-Sobolev norm, as in
the previous proof, is an immediate consequence of Theorem~\ref{T2} -- our assumptions
on $P$ give us an embedding of $W^{1,P}(M,N)$ into $H^{1,n-1}(M,N)$.

What is left to prove is that if the universal cover of $N$ is a rational homology
sphere, the degree cannot be defined continuously in $W^{1,P}(M,N)$.
If $F_k=G\circ g_k\circ F$ are the mappings defined as in the proof of Theorem~\ref{T2},
then $F_k$ converges a.e. to a constant map, so it converges in $L^P$. As
concerns the derivative, $|DF_k|$ is bounded by $Ck$ and
different than zero on a set $F^{-1}(C_k)$ whose measure is comparable to $k^{-n}$, so
$$
\int_M P(|DF_k|)\leq P(Ck) |F^{-1}(C_k)|\leq C' P(k)k^{-n}\to 0
\quad
\mbox{as $k\to \infty$}
$$
by \eqref{E1.5} and the doubling condition.
Hence $DF_k$ converges to zero in mean and thus in $L^P$.
\end{proof}

\section*{Acknowledgements}
The authors would like to thank all who participated in a
discussion posted on Algebraic Topology Discussion List
(http://www.lehigh.edu/$\sim$dmd1/post02.html) in the thread {\em
Question about rational homology spheres} posted by Haj\l{}asz in
2002. The authors also thank J. K\c{e}dra for pointing
Proposition~\ref{cp2tocp1}. At last, but not least, the authors
acknowledge the kind hospitality of CRM at the Universitat
Aut\`onoma de Barcelona, where the research was partially carried
out.

\end{document}